\documentclass[12pt]{pub}
\usepackage{graphicx}
\usepackage{mathrsfs}
\usepackage{amsfonts,amsmath,amsthm, amssymb}
\usepackage{latexsym, euscript, epic, eepic, color}
\renewcommand{\theequation}{\thesection.\arabic{equation}}
\newtheorem{lemma}{Lemma}[section]
\newtheorem{theorem}{Theorem}[section]
\newtheorem{definition}{Definition}[section]
\newtheorem{proposition}{Proposition}[section]
\newtheorem{corollary}{Corollary}

\newtheorem{remark}{Remark}

\begin{document}
\title[Large and moderate deviation principles for Engel continued fractions\dots]{Large and moderate deviation principles for Engel continued fractions}

\author{\small Lulu Fang}
\address{\tiny School of Mathematics,
South China University of Technology,
Guangzhou 510640 P.R. China}
\email{\scriptsize f.lulu@mail.scut.edu.cn}

\author{\small Min Wu}
\address{\tiny School of Mathematics,
South China University of Technology,
Guangzhou 510640 P.R. China}
\email{\scriptsize wumin@scut.edu.cn}

\author{Lei Shang}
\address{\tiny School of Mathematical Science,
Anhui University,
Hefei 230601, P.R. China}
\email{\scriptsize auleishang@gmail.com}


\begin{abstract}
Large and moderate deviation principles are proved for Engel continued fractions, a new type of continued fraction expansion with non-decreasing partial quotients in number theory.
\end{abstract}
\maketitle

\section{Introduction}

Given a real number, there are various ways to represent it as an expansion of digits, such as continued fractions (see Khintchine \cite{lesKhi64}) and series expansions (see Galambos \cite{lesGal76}) including $\beta$-expansions \cite{lesRen57}, L\"{u}roth expansions \cite{lesLur83},
Engel expansions \cite{lesE.R.S58} and alternating Engel expansions \cite{lesSha86} and so on. Perhaps the most well-known representation of real numbers is continued fractions. Over the last thirty years, considerable interests are shown in various continued fraction expansions.
Examples of such continued fraction expansions include, for example, backward continued fractions \cite{lesAF84}, $a/b$-continued fractions \cite{lesDKL15}, Oppenheim continued fractions \cite{lesFWW07}, multidimensional continued fractions \cite{lesHK01}, $\alpha$-continued fractions \cite{lesNN02} and Rosen continued fractions \cite{lesDKS09, lesRos54}. Most of these continued fraction expansions have invariant and ergodic measures which are absolutely continuous with respect to Lebesgue measure. By means of the ergodic theory and probability theory, some metric and statistical properties of these continued fraction expansions have been well studied when the real number is sampled from the uniform distribution in some interval. Besides, series expansions of real numbers also gave rise to a fruitful study of their ergodic and statistical properties. The history of this topic may begin with Borel, Sierpi\'{n}ski, Kuzmin, L\'{e}vy and continue with the seminal work of the Hungarian school (Erd\H{o}s, Tur\'{a}n, R\'{e}nyi, Sz\"{u}sz, R\'{e}v\'{e}sz, Galambos). Such a theory of using probability theory to answer questions of number theory is called the \emph{probabilistic number theory}, see Elliott \cite{lesEll80}, R\'{e}nyi \cite{lesRen58} and Tenenbaum \cite{lesTen95}.
Essentially these classical studies for continued fractions and series expansions
mainly focus on the distribution law, the law of large numbers, the central limit theorem and the law of the iterated logarithm for the digit sequence occurring in these expansions.
However, it is worth pointing out that these classical limit theorems basically concern that the averages taken over large samples converge to expected values in some sense, but say little or nothing about the rate of convergence. One way to address this is the theory of large deviations in modern probability theory.

Let $(\Omega, \mathcal{F}, \mathbf{P})$ be a probability space and $\{X_n: n \geq 1\}$ be a sequence of real-valued random variables defined on $(\Omega, \mathcal{F}, \mathbf{P})$.
A function $I: \mathbb{R} \to [0,\infty]$ is called a \emph{good rate function} if it is lower semi-continuous and has compact level sets.
We say that the sequence $\{X_n: n \geq 1\}$ satisfies a \emph{large deviation principle} (LDP for short) with speed $n$ and good rate function $I$ if for any Borel set $\Gamma$,
\begin{equation*}
-\inf_{x \in \Gamma^\circ}I(x) \leq \liminf_{n \to \infty} \frac{1}{n}\log \mathbf{P}(X_n \in \Gamma) \leq \limsup_{n \to \infty} \frac{1}{n}\log \mathbf{P}(X_n \in \Gamma) \leq -\inf_{x \in \overline{\Gamma}}I(x),
\end{equation*}
where $\Gamma^\circ$ and $\overline{\Gamma}$ denotes the interior and the closure of $\Gamma$ respectively. We use the notation $\mathbf{E}(\xi)$ to denote the expectation of a random variable $\xi$ with respect to the probability measure $\mathbf{P}$.~G\"{a}rtner-Ellis theorem (see Theorem 2.3.6 in \cite{lesD.Z1998}) tells us that the rate function is often given in term of the Legendre transform of the pressure function. That is, $I(x) = \sup_{\theta \in \mathbb{R}}\left\{\theta x - \Lambda(\theta)\right\}$,
where the \emph{pressure function} $\Lambda(\cdot)$ is defined as
\begin{equation}
\Lambda(\theta): = \lim_{n \to \infty} \frac{1}{n} \log \mathbf{E}(e^{n\theta \cdot X_n})
\end{equation}
for any $\theta \in \mathbb{R}$ when it exists.
For instance, we can study the probability that the empirical mean of a sequence of random variables deviates away from its ergodic mean. These probabilities are exponentially small in general and follow the large deviation principle.
Formally, there is no distinction between the large deviation principle and the moderate deviation principle (MDP for short). Usually LDP characterizes the convergence speed of the law of large numbers, while MDP describes the speed of convergence between the law of large numbers and the central limit theorem.
For more details about large and moderate deviations, we refer the reader to Dembo and Zeitouni \cite{lesD.Z1998}, Touchette \cite{lesTou09} and Varadhan \cite{lesVar1984}.

The theory of large deviations plays an important role in the framework of the stochastic processes. Nowadays, it has been developed quite rapidly in different directions and many applications arise, for example, in mathematics, statistics, computer science, physics, and in other fields.
In this paper, we will investigate the large and moderate deviation principles for Engel continued fractions (see Section 2) which is a new type of continued fraction expansion with non-decreasing partial quotients.
Recently, such kinds of problems have been considered by Zhu \cite{lesZhu2014}, Hu \cite{lesHu2015}, Fang \cite{lesFangSPL, lesFangJNT}, and Fang and Wu \cite{lesFW}. Zhu studied the large deviations for Engel expansions (see \cite{lesE.R.S58}), whose digit sequence is also non-decreasing.
Erd\H{o}s et al.~\cite{lesE.R.S58} proved that the digit sequence of Engel expansions forms a time-homogeneous Markov chain with explicit transition probability functions. And they also gave an explicit computation for the distribution of the digit and its asymptotic analysis.
These properties play an important role in Zhu's work.
However, the Engel continued fraction system is not Markovian (see Remark \ref{NMC}). That is where the difficulty is.
In \cite{lesFangSPL}, the author studied the large deviations for modified Engel continued fractions. This new continued fraction system is a small modification of Engel continued fractions but its partial quotient sequence is strictly increasing.
In \cite{lesFangJNT}, the author established the large deviations for alternating Engel expansions. This alternating expansion can be treated as the Engel expansion with alternating terms while its digit sequence is strictly increasing. Although these two expansions share the same classical limit theorems (see \cite{lesV.B.P99, lesWil73}), the author also remarked that there is a difference between these expansions in the context of large deviations.
Here we remark that the rate functions of the large deviations in \cite{lesFangSPL} and \cite{lesFangJNT} are same but different from Zhu's. The authors in \cite{lesFW} considered two interesting discrete Markov processes introduced by Williams \cite{lesWil73}, which share the same classical limit theorems
but have a difference in the context of large deviations.
Besides, the theory of large deviations also has been applied to the analytic number theory, see F\'{e}ray et al.~\cite{lesFMN}, Giuliano and Macci \cite{lesGM15},
Mehrdad and Zhu \cite{lesMZ16, lesMZar} and Radziwill \cite{lesRad09}. It seems that the large and moderate deviations might have the potential to become the useful tools in studying probabilistic and analytic number theory.

This paper is organized as follows. Section 2 is devoted to stating the large and moderate deviation principles for Engel continued fractions.
In Section 3, we recall the basic properties of the Engel continued fraction expansion and show that the partial quotient sequence of Engel continued fractions
 is not Markovian (see Remark \ref{NMC}) but very close to a homogeneous Markov chain (see Proposition \ref{Markov chain}), which plays an important role in our proofs. The proofs of our main results are given in Section 4. In Section 5,  we will see that the theory of large deviations is a more useful tool of dealing with such questions than other tools (such as, fractal) in some sense.

\section{Main results}

In 2002, Hartono et al.~\cite{lesHKS02} introduced a new continued fraction algorithm with non-decreasing partial quotients, called the \emph{Engel continued fraction} (ECF, for short) expansion. The name of this new continued fraction expansion is borrowed from the classical Engel expansion (see Erd\H{o}s et al.~\cite{lesE.R.S58}). Now we give the algorithm of this new continued fraction expansion. Let $T_E: (0,1] \longrightarrow (0,1]$ be the \emph{ECF map} given by
\[
T_E x = \frac{1}{\left[\frac{1}{x}\right]} \cdot \left(\frac{1}{x}-\left[\frac{1}{x}\right]\right),
\]
where $[x]$ denotes the greatest integer not exceeding $x$. Then every real number $x \in (0,1]$ can be uniquely written as
\begin{equation}\label{ECF}
x = \dfrac{1}{b_1(x) +\dfrac{b_1(x)}{b_2(x) + \ddots +\dfrac{b_{n-1}(x)}{b_n(x)+ \ddots}}},
\end{equation}
where $b_1(x) = [1/x] \in \mathbb{N}$ and $b_{n +1}(x) = b_1(T_E^nx)$ with $b_{n+1}(x) \geq b_n(x)$ for all $n \geq 1$. The representation (\ref{ECF}) is said to be the \emph{ECF expansion} of $x$ and $b_n(x)$ are called the \emph{partial quotients} of the ECF expansion of $x$ ($n \in \mathbb{N}$). Sometimes we write the form (\ref{ECF}) as $[[b_1(x),b_2(x), \cdots, b_n(x),\cdots]]$. Hartono et al.~\cite{lesHKS02} studied the arithmetic and ergodic properties of $T_E$ associated to this new continued fraction expansion. Moreover, they showed that $T_E$ is ergodic with respect to Lebesgue measure and proved that $T_E$ has no finite invariant measure equivalent to the Lebesgue measure, but has infinitely many $\sigma$-finite nvariant measures.
This new class of continued fractions can be applied to designing a pseudo random bit generator (PRBG) and proposing a new scheme for image cryptosystems, see Masmoudi et al. \cite{lesMPB10, lesMBP12}.

Now we denote $(\Omega, \mathcal{F}, \mathbf{P})$ by the probability space, where $\Omega  = (0,1]$, $\mathcal{F}$ is the Borel $\sigma$-algebra on $(0,1]$ and $\mathbf{P}$ denotes the Lebesgue measure on $(0,1]$. Kraaikamp and Wu \cite{lesKW04} proved a strong law of large numbers for the partial quotient sequence $\{b_n: n \geq 1\}$, i.e., for $\mathbf{P}$-almost surely $x \in (0,1]$,
\begin{equation*}
\lim_{n \to \infty} \frac{1}{n} \log b_n(x) = 1.
\end{equation*}
This implies that the probability of the event that $\frac{\log b_n(x)}{n}$ deviates away from its ergodic mean 1 tends to zero as $n$ goes to infinity. However, it does not give this decay a speed. This leads to the study of large deviations for the ECF expansion.

\begin{theorem}\label{LDP}
Let $\{b_n: n \geq 1\}$ be the partial quotient sequence of the ECF expansion. Then $\left\{\frac{\log b_n - n}{n}: n \geq 1\right\}$ satisfies a LDP with speed $n$ and good rate function
\begin{equation}\label{LDP equation}
I(x)=
\begin{cases}
x - \log (x +1),  &\text{if $x >-\frac{\sqrt{5}-1}{2}$};\\
-\frac{\sqrt{5}+1}{2}(x +1)+ 2 \log \frac{\sqrt{5}+1}{2},  &\text{if $ -1 \leq x \leq-\frac{\sqrt{5}-1}{2} $};\\
+\infty,  & \text{otherwise}.
\end{cases}
\end{equation}
\end{theorem}

\begin{remark}
As we have pointed out at the end of introduction, the rate functions of the large deviations in \cite{lesFangSPL} and \cite{lesFangJNT} are same but different from Zhu's and also different from ours.
What is interesting though, is that the rate functions of the large deviations for Engel continued fractions and Engel expansions are similar in structure (comparing our Theorem \ref{LDP} with Theorem 1.2 of Zhu \cite{lesZhu2014}).
We think the main reason is that the digit sequence of modified Engel continued fractions or alternating Engel expansions is strictly increasing; while the digit sequence of Engel expansions or Engel continued fractions is non-decreasing (see Remark \ref{vs} for more details).
\end{remark}


\begin{remark} If we consider the large deviations for the sequence $\left\{\frac{\log b_n - n}{n}: n \geq 1\right\}$ with $b_1 \geq b~(b \in \mathbb{N})$,
we will obtain that it also satisfies a LDP with speed $n$ and good rate function
\begin{equation*}
I_b(x)=
\begin{cases}
x - \log (x +1),  &\text{if $x >-1+\xi^{-1}_b$};\\
(1-\xi_b)(x +1)+ \log \xi_b,  &\text{if $ -1 \leq x \leq -1+\xi^{-1}_b$};\\
+\infty,  & \text{otherwise},
\end{cases}
\end{equation*}
where $\xi_b= (b^2+2+\sqrt{b^2+4b})/(2b)$. This indicates that the initial value of $b_1$ will affect the rate function of large deviations for the ECF expansion. Moreover, it is easy to check $I_1(x) = I (x)$ by (\ref{LDP equation}) and $I_b(x) \to I_\infty (x)$  as $b \to \infty$, where $I_\infty$ is defied as
\[
I_\infty(x)=
\begin{cases}
x - \log (x +1),  &\text{if $x >-1$};\\
+\infty,  & \text{otherwise}.
\end{cases}
\]
Here we emphasize that $I_\infty$ is the rate function of large deviations for modified ECF expansions \cite{lesFangSPL} or alternating Engel expansions \cite{lesFangJNT} or the empirical mean of independent and identically distributed exponential random variables with parameter 1 (see \cite[Exercise 2.2.23]{lesD.Z1998} and \cite[Example 3.2]{lesTou09}).
\end{remark}

As an application of Theorem \ref{LDP}, we obtain that the Lebesgue measure of the set of points $x \in (0,1]$ for which $\log b_n(x)/n$ deviates away from its ergodic mean $1$ decays to zero exponentially as $n$ tends to infinity.

\begin{corollary}\label{Cor}
For any $\varepsilon > 0$, there exist two positive constants $\alpha$ and $\beta$ (both only depending on $\varepsilon$) such that for all $n \geq 1$, we have
\[
\mathbf{P}\left\{x \in (0,1]:\left|\frac{\log b_n(x)}{n}- 1\right| \geq \varepsilon\right\}\leq \alpha e^{-\beta n}.
\]
\end{corollary}

\begin{remark}
By Borel-Cantelli lemma, we deduce easily that this result implies the strong law of large numbers for the sequence of partial quotients of ECF expansions obtained by Kraaikamp and Wu \cite{lesKW04} in 2004.
\end{remark}

In 2007, Fan et al.~\cite{lesFWW07} established a central limit theorem for the  partial quotient sequence $\{b_n: n \geq 1\}$, which tells us by how much the quantity $\log b_n(x)$ normally exceeds its ergodic mean 1, namely by an order of $\sqrt{n}$. That is, for every $y \in \mathbb{R}$,
\begin{equation*}
\lim_{n \to \infty} \mathbf{P} \big\{x \in (0,1]:\log b_n(x) - n \geq \sqrt{n}y\big\}= 1-\frac{1}{\sqrt{2\pi}}\int_{-\infty}^y e^{-t^2/2}dt.
\end{equation*}
On the other hand, we have seen that for any $\varepsilon > 0$, the probability
\[
\mathbf{P}\big\{x \in (0,1]:\log b_n(x)- n \geq n\varepsilon\big\}
\]
decays to zero exponentially as $n$ tends to infinity. These two results imply that for any positive sequence $a_n$ with $\sqrt{n} \ll a_n \ll n$, we still have
\[
\mathbf{P}\big\{x \in (0,1]:\log b_n(x)- n \geq a_n\big\}
\]
tends to zero as $n$ goes to infinity. However, neither the above central limit theorem nor the Theorem \ref{LDP} tells us how fast this convergence is. This question is in the remit of the moderate deviation principle for ECF expansions stated below.

\begin{theorem}\label{MDP}
Let $\{b_n: n \geq 1\}$ be the partial quotient sequence of the ECF expansion and $\{a_n: n \geq 1\}$ be a positive sequence satisfying
\begin{equation}\label{a-n}
a_n \to \infty, \ \ \ \ \ \frac{a_n}{\sqrt{n\log n}} \to \infty \ \ \ \ \  and \ \ \ \ \ \frac{a_n}{n} \to 0.
\end{equation}
Then $\left\{\frac{\log b_n - n}{a_n}: n\geq 1\right\}$ satisfies an MDP with speed $n^{-1}a_n^2$ and good rate function $J(x)= x^2/2$ for any $x \in \mathbb{R}$.
\end{theorem}

\begin{remark}
We may obtain this MDP result if the second condition of $a_n$ in (\ref{a-n}) is replaced by $a_n /\sqrt{n} \to \infty$ as $n \to \infty$. However, we need the condition $\frac{a_n}{\sqrt{n\log n}} \to \infty$ to avoid some technical difficulties. The main reason is that we cannot find some finer estimates when we make use of the G\"{a}rtner-Ellis theorem. For this reason, we believe that
\[
\liminf_{n \to \infty} \frac{\log b_n(x) - n}{\sqrt{2n \log\log n}} =-1\ \ \ \ \text{and}\ \ \ \ \limsup_{n \to \infty} \frac{\log b_n(x) - n}{\sqrt{2n \log\log n}} =1
\]
hold for $\mathbf{P}$-almost surely $x \in (0,1]$ since the law of the iterated logarithms can be obtained by using some moderate deviation inequalities (see \cite{lesBCR06, lesChen07, lesGao08}).
\end{remark}

As an application of Theorem \ref{MDP}, we immediately get the following corollary.
\begin{corollary}
Let $\{b_n: n \geq 1\}$ be the partial quotient sequence of the ECF expansion and $a_n = n^p$ with $p \in (1/2,1)$. Then $\left\{\frac{\log b_n - n}{n^p}: n\geq 1\right\}$ satisfies an MDP with speed $n^{2p-1}$ and good rate function $J(x)= x^2/2$ for any $x \in \mathbb{R}$.
\end{corollary}

\section{Preliminary}

In this section, we recall some definitions and several arithmetic and statistical properties of the ECF expansion.
We first give an elementary arithmetic property of the ECF expansion in representation of real numbers, which is obtained by Hartono, Kraaikamp and Schweiger \cite{lesHKS02}.

\begin{proposition}(\cite[Theorem 2.1]{lesHKS02})\label{real number}
Let $x \in (0,1]$ be a real number. Then $x$ has a finite ECF expansion (i.e. $T_E^n x =0$ for some $n \geq 1$) if and only if $x$ is rational.
\end{proposition}

\begin{definition}
An $n$-block $(b_1, b_2, \cdots, b_n)$ is said to be admissible for ECF expansions if there exists $x \in (0,1]$ such that $b_j(x) = b_j$ for all $1 \leq j \leq n$. An infinite sequence $(b_1, b_2, \cdots, b_n, \cdots)$ is called an admissible sequence if $(b_1, b_2, \cdots, b_n)$ is admissible for any $n \geq 1$.
\end{definition}

The following proposition, due to Fan, Wang and Wu \cite{lesFWW07}, gives a characterization of all admissible sequences occurring in the ECF expansion.

\begin{proposition}(\cite[Proposition 2.2]{lesFWW07})\label{characterization}
A sequence of positive integers $(b_1, b_2, \cdots, b_n, \cdots)$ is admissible for ECF expansions if and only if for all $n \geq 1$,
\[
b_{n+1} \geq b_n.
\]
\end{proposition}

\begin{definition}
Let $(b_1, b_2, \cdots, b_n)$ be an admissible sequence. We call the set
\begin{equation*}
B(b_1, b_2, \cdots, b_n) = \big\{x \in (0,1]: b_1(x)=b_1,b_2(x)=b_2,\cdots,b_n(x)=b_n\big\}
\end{equation*}
the $n$-th order cylinder. In other words, it is the set of points beginning with $(b_1,b_2 \cdots, b_n)$ in their ECF expansions.
\end{definition}

The following proposition is about the structure and the length of cylinders, which has been obtained in \cite{lesHKS02} (see also \cite{lesFWW07}).

\begin{proposition}\label{cylinder}
Let $(b_1, b_2, \cdots, b_n)$ be an admissible sequence. Then the $n$-th order cylinder
$B(b_1, b_2, \cdots, b_n)$ is an interval with two endpoints
\[
[[b_1,\cdots,b_{n-1}, b_n]]\ \ \ \  \text{and}\ \ \ \ [[b_1,\cdots,b_{n-1}, b_n+1]].
\]
Hence that for all $n \geq 1$,
\begin{equation}\label{cylinder length}
\mathbf{P}(B(b_1, b_2, \cdots, b_n))= \frac{\prod_{i=1}^{n-1}b_i}{Q_n(Q_n+Q_{n-1})},
\end{equation}
where the quantity $Q_n$ satisfies the recursive formula $Q_n = b_n Q_{n-1}+ b_{n-1}Q_{n-2}$ under the conventions $Q_{-1} =0$ and $Q_0 =1$.
\end{proposition}

\begin{remark}\label{NMC}
The Markov property states that the distribution of the forthcoming state depends only on the current state and does not depend on the previous ones.
So the partial quotient sequence $\{b_n: n \geq 1\}$ does not form a homogeneous Markov chain. In fact, by (\ref{cylinder length}), we obtain that $\mathbf{P} (B(1,1,2)) = 1/35$, $\mathbf{P} (B(1,2,2)) = 1/44$, $\mathbf{P} (B(2,2,2)) = 1/88$, $\mathbf{P} (B(1,1,2,2)) = 1/133$,
$\mathbf{P} (B(1,2,2,2)) = 1/165$, $\mathbf{P} (B(2,2,2,2)) =1/330$ and hence that
\[
\mathbf{P}(b_4= 2~|~b_3=2, b_2=1, b_1=1) = 5/19 \neq 972/3667 = \mathbf{P}(b_4= 2~|~b_3=2).
\]
\end{remark}

Although the partial quotient sequence $\{b_n: n \geq 1\}$ does not form a homogeneous Markov chain, we have

\begin{proposition} \label{Markov chain}
Let $\{b_n: n \geq 1\}$ be the partial quotient sequence of the ECF expansion. Then for any $k \geq j \geq 1$,
\begin{equation}\label{initial distribution}
\mathbf{P}(b_1 = j) = \frac{1}{j(j+1)}
\end{equation}
and the conditional probabilities
\begin{equation}\label{transition distribution}
\frac{j}{k(k+2)} \leq \mathbf{P}(b_{n+1}= k~|~b_n=j) \leq \frac{j+1}{k(k+1)}\ \ \ \text{for all}\ n \geq 1.
\end{equation}
\end{proposition}

\begin{proof}
The equation (\ref{initial distribution}) is obvious by taking $n=1$ in (\ref{cylinder length}). Now we prove the inequalities (\ref{transition distribution}). For two integers $1 \leq a \leq b$ and real number $0 \leq y <1$, we define
\[
\Phi(a,b,y)= \frac{a(1+y)}{(b+ay)(b+1+ay)}.
\]
For any $n \in \mathbb{N}$ and admissible sequence $(b_1,\cdots, b_n,d_{n+1})$, by (\ref{cylinder length}),
we deduce that
\begin{equation}\label{lei Markov}
\frac{\mathbf{P}(B(b_1,\cdots, b_n,b_{n+1}))}{\mathbf{P}(B(b_1, \cdots, b_n))} = \frac{b_nQ_n(Q_n + Q_{n-1})}{Q_{n+1}(Q_{n+1} + Q_n)} = \Phi(b_n, b_{n+1}, y_n),
\end{equation}
where $y_n = Q_{n-1}/Q_n \geq 0$ and the last equality is from the recursive formula $Q_n = b_n Q_{n-1}+ b_{n-1}Q_{n-2}$. This recursive formula of $Q_n$ indicates that $0 \leq b_ny_n \leq 1$. So,
\[
\Phi(b_n, b_{n+1}, y_n) \leq \frac{b_n+1}{b_{n+1}(b_{n+1} +1)}.
\]
Since $b_{n+1} \geq b_n$ and $0 \leq b_ny_n \leq 1$, we obtain that
\[
\Phi(b_n, b_{n+1}, y_n) \geq \frac{b_n}{b_{n+1}(b_{n+1} +2)}.
\]
Combing these with (\ref{lei Markov}), we have that
\[
\frac{b_n}{b_{n+1}(b_{n+1} +2)} \leq \frac{\mathbf{P}(B(b_1,\cdots, b_n,b_{n+1}))}{\mathbf{P}(B(b_1, \cdots, b_n))} \leq \frac{b_n+1}{b_{n+1}(b_{n+1} +1)}.
\]
Taking summations for all admissible sequences $(b_1,\cdots,b_{n})$, the inequalities (\ref{transition distribution}) are established.
\end{proof}

\section{Proofs of main results}
In this section, we will give the proof of Theorems \ref{LDP} and \ref{MDP}.
To do this we first give a simple but useful lemma, which states roughly that the rate of growth for a finite
sum of sequences equals the maximal rate of growth of the summands.

\begin{lemma}\label{basic}
Let $m \geq 2$ be an integer and $\{c_n\}_{n \geq 1}$ be a sequence satisfying $c_n \to \infty$ as $n \to \infty$. Then
\[
\limsup_{n \to \infty} \frac{1}{c_n} \log \left(\sum_{k=1}^m d^{(k)}_n\right) = \max_{1 \leq k \leq n} \limsup_{n \to \infty} \frac{1}{c_n} \log d^{(k)}_n,
\]
where $\{d^{(1)}_n\}_{n \geq 1}, \cdots, \{d^{(m)}_n\}_{n \geq 1}$ are non-negative sequences.
\end{lemma}

\begin{proof}
Let $m \geq 2$ be an integer. For any $n \geq 1$, since
\[
0 \leq \log \left(\sum_{k=1}^m d^{(k)}_n\right)- \max_{1 \leq k \leq n}\log d^{(k)}_n \leq \log m,
\]
we deduce that
\[
\limsup_{n \to \infty} \frac{1}{c_n} \log \left(\sum_{k=1}^m d^{(k)}_n\right) = \limsup_{n \to \infty} \frac{1}{c_n} \max_{1 \leq k \leq n}\log d^{(k)}_n
= \max_{1 \leq k \leq n} \limsup_{n \to \infty} \frac{1}{c_n} \log d^{(k)}_n.
\]
\end{proof}

\begin{lemma}\label{jichu}
Let $\theta <1$ be a real number. Then for any $j>1$, we have
\[
\sum_{k=j}^{\infty} \frac{j}{k(k+2)} \left(\frac{k}{j}\right)^\theta \geq \frac{j}{j+2} \cdot \frac{1}{1-\theta}
\]
and
\[
\sum_{k=j}^{\infty} \frac{j+1}{k(k+1)} \left(\frac{k}{j}\right)^\theta \leq  \left(1+ \frac{1}{j}\right)\cdot \left(1-\frac{1}{j}\right)^{\theta -1}\cdot \frac{1}{1-\theta}.
\]
\end{lemma}

\begin{proof}
Let $\theta <1$ be a real number. For any $j > 1$, we obtain that
\begin{equation}\label{xiaoyu1}
\sum_{k =j}^{\infty}\frac{1}{k^{2-\theta}} \geq \int_j^{\infty} \frac{1}{x^{2-\theta}}dx = \frac{1}{1-\theta}\cdot j^{\theta-1}.
\end{equation}
Since
\begin{equation*}
\frac{k^{\theta}}{k(k+2)} = \frac{1}{k^{2-\theta}} \cdot\frac{k}{k+2}\ \ \ \text{and}\ \ \ \frac{k}{k+2} \geq \frac{j}{j+2} \ \ \text{for any}\ k \geq j,
\end{equation*}
we have that
\[
\sum_{k=j}^{\infty}\frac{j}{k(k+2)}\left(\frac{k}{j}\right)^\theta = \frac{j}{j^{\theta}}\cdot \sum_{k =j}^{\infty}\frac{k^{\theta}}{k(k+2)} \geq \frac{j}{j^{\theta}} \cdot \frac{j}{j+2} \cdot \sum_{k =j}^{\infty} \frac{1}{k^{2-\theta}}.
\]
Combing this with (\ref{xiaoyu1}), we deduce that
\[
\sum_{k=j}^{\infty}\frac{j}{k(k+2)}\left(\frac{k}{j}\right)^\theta \geq \frac{j}{j^{\theta}} \cdot \frac{j}{j+2} \cdot \frac{1}{1-\theta}\cdot j^{\theta-1} = \frac{j}{j+2} \cdot \frac{1}{1-\theta}.
\]

Notice that for any $j >1$,
\begin{equation*}\label{xiaoyu}
\sum_{k =j}^{\infty}\frac{k^{\theta}}{k(k+1)} \leq \sum_{k =j}^{\infty}\frac{1}{k^{2-\theta}}
\leq \int_{j-1}^{\infty} \frac{1}{x^{2-\theta}}dx =  \frac{1}{1-\theta}\cdot(j-1)^{\theta-1},
\end{equation*}
we know that
\[
\sum_{k=j}^{\infty} \frac{j+1}{k(k+1)} \left(\frac{k}{j}\right)^\theta = \frac{j+1}{j^\theta} \cdot \sum_{k =j}^{\infty}\frac{k^{\theta}}{k(k+1)} \leq \left(1+ \frac{1}{j}\right)\cdot \left(1-\frac{1}{j}\right)^{\theta -1}\cdot \frac{1}{1-\theta}.
\]
\end{proof}

\subsection{Proof of large deviation principle}

For any $m,n \in \mathbb{N}$, we define
\[
\Sigma_{n,m} = \left\{(b_1,\cdots,b_{n-1},m): (b_1,\cdots,b_{n-1},m)\ \text{is admissible}\right\}
\]
and
\[
\Sigma_{n,\leq m} = \left\{(b_1,\cdots,b_{n-1},j): (b_1,\cdots,b_{n-1},j)\ \text{is admissible for all}\ 1 \leq j \leq m\right\}.
\]
We believe that the following result is known in the combinatorial theory; while, to our knowledge, we cannot find its proof in any book about combination. For the completeness of this paper, we give its proof here using the mathematical induction.

\begin{lemma}\label{Numbers}
For any $m,n \in \mathbb{N}$,
\[
\sharp \Sigma_{n,m} = C^{m-1}_{n+m-2}\ \ \ \ \ \text{and}\ \ \ \ \ \sharp \Sigma_{n,\leq m} = C^{m-1}_{n+m-1},
\]
where $\sharp$ denotes the number of elements of a finite set and $C^m_n$ means the number of possible combinations of $m$ objects from a set of $n$ objects.
\end{lemma}

\begin{proof}
We first prove the first formula by induction and then prove the second formula.
In view of Proposition \ref{characterization}, we know that the set $\Sigma_{n,1}$ only has one element as $(1,1,\cdots,1)$. So, $\sharp \Sigma_{n,1} = 1$. It is easy to check that the set $\Sigma_{n,2}$ has $n$ elements like $(1,\cdots,1,2), (1,\cdots,2,2), \cdots, (2,\cdots,2,2)$. Thus, $\sharp \Sigma_{n,2} =n$. These indicate that the formula $\sharp \Sigma_{n,m} = C^{m-1}_{n+m-2}$ is true for $m=1$ and $2$. Now we assume that this formula is true for all $m \leq k$, i.e.,
\[
\sharp \Sigma_{n,1} = C^0_{n-1}, \sharp \Sigma_{n,2} = C^1_{n}, \sharp \Sigma_{n,3} = C^2_{n+1}, \cdots, \sharp \Sigma_{n,k} = C^{k-1}_{n+k-2}.
\]
For $m = k+1$, notice that
\[
\sharp \Sigma_{n,k+1} = \sharp \Sigma_{n-1,1} + \sharp \Sigma_{n-1,2} + \sharp \Sigma_{n-1,3} + \cdots + \sharp \Sigma_{n-1,k} + \sharp \Sigma_{n-1,k+1},
\]
we obtain that
\begin{align}\label{die}
\sharp \Sigma_{n,k+1} - \sharp \Sigma_{n-1,k+1} &= \sharp \Sigma_{n-1,1}+\sharp \Sigma_{n-1,2} + \sharp \Sigma_{n-1,3} + \cdots + \sharp \Sigma_{n-1,k}\notag \\
& = C^0_{n-1} + C^1_{n} + C^2_{n+1} + \cdots + C^{k-1}_{n+k-3} \notag \\
& = C^0_{n} + C^1_{n} + C^2_{n+1} + \cdots + C^{k-1}_{n+k-3} \notag \\
& = C^1_{n+1}+ C^2_{n+1} + \cdots + C^{k-1}_{n+k-3}\notag \\
& = \cdots \notag \\
&= C^{k-1}_{n+k-2},
\end{align}
where third equality follows from  $C^0_{n-1} = C^0_{n} =1$ and the forth equality is from the basic combination equation $C^j_{n} + C^{j+1}_{n} = C^{j+1}_{n+1} $. Being similar to (\ref{die}), we deduce that
\[
\sharp \Sigma_{n-1,k+1} - \sharp \Sigma_{n-2,k+1} =  C^{k-1}_{n+k-3}, \cdots, \sharp \Sigma_{2,k+1} - \sharp \Sigma_{1,k+1} = C^{k-1}_{k}.
\]
Since $C^j_n = C^{n-j}_n$, we have that
\begin{align*}
\sharp \Sigma_{n,k+1} &= C^{k-1}_{n+k-2} + C^{k-1}_{n+k-3}+ \cdots + C^{k-1}_{k+1}+ C^{k-1}_{k} + \sharp \Sigma_{1,k+1} \\
& = C^{n-1}_{n+k-2} + C^{n-2}_{n+k-3}+ \cdots + C^{2}_{k+1}+ C^{1}_{k} + 1\\
& = C^{n-1}_{n+k-2} + C^{n-2}_{n+k-3}+ \cdots + C^{2}_{k+1} + C^{1}_{k+1}\\
& = \cdots\\
&= C^{n-1}_{n+k-1} = C^{k}_{n+k-1},
\end{align*}
where the second equality follows from $\Sigma_{1,k+1} = \{(k+1)\}$ with only one element and the third equality is also from the basic combination equation $C^j_{n} + C^{j+1}_{n} = C^{j+1}_{n+1} $. This is to say, the formula $\sharp \Sigma_{n,m} = C^{m-1}_{n+m-2}$ is also true for $m=k+1$. By induction, we obtain the desired result. Now we are ready to prove $\sharp \Sigma_{n,\leq m} = C^{m-1}_{n+m-1}$. In fact, it follows from the first result that
\begin{align*}
\sharp \Sigma_{n,\leq m} &= \sharp \Sigma_{n,1} + \sharp \Sigma_{n,2} + \sharp \Sigma_{n,3}+ \cdots + \sharp \Sigma_{n,m}\\
& = C^0_n + C^1_{n+1} + C^2_{n+2} + \cdots + C^{m-1}_{n+m-2} \\
&=  C^1_{n+1} + C^2_{n+2} + \cdots + C^{m-1}_{n+m-2}\\
& = \cdots \\
&= C^{m-1}_{n+m-1}.
\end{align*}
\end{proof}

\begin{lemma}\label{number}
Let $\{b_n: n \geq 1\}$ be the partial quotient sequence of the ECF expansion. Then for any $n \in \mathbb{N}$ and $j >1$,
\[
 \mathbf{P}(b_n =1)\geq A^{-1}\cdot \left(\frac{\sqrt{5}+1}{2}\right)^{-2n}\ \text{and}\ \ \ \ \mathbf{P}(b_n \leq j) \leq A\cdot C^{j-1}_{n+j-1} \cdot  \left(\frac{\sqrt{5}+1}{2}\right)^{-2n},
\]
where $A >1$ is an absolute constant and the combinatorial number $C^m_n$ is as defined in Lemma \ref{Numbers}.
\end{lemma}

\begin{proof}
By the non-decreasing property of $b_n$, it follows from Proposition \ref{cylinder} that
\begin{equation}\label{number 1}
\mathbf{P}(b_n = 1) = \mathbf{P}(B(\underbrace{1,1,\cdots,1}_{n})) = \frac{1}{Q_n(Q_n+Q_{n-1})},
\end{equation}
where $Q_n$ satisfies the recursive formula $Q_n = Q_{n-1}+ Q_{n-2}$ under the conventions $Q_{-1} =0$ and $Q_0 =1$. This indicates that $\{Q_n\}_{n\geq 0}$ is a sequence of Fibonacci numbers. So,
\[
Q_n =\frac{1}{\sqrt{5}}\cdot \left(\left(\frac{1+\sqrt{5}}{2}\right)^{n+1}- \left(\frac{1-\sqrt{5}}{2}\right)^{n+1}\right). 
\]
Since $-1 \leq (\frac{1-\sqrt{5}}{1+\sqrt{5}})^{n} \leq \frac{\sqrt{5}-1}{\sqrt{5}+1}$ for any $n \geq 1$, it is easy to check that
\[
\frac{2}{5+\sqrt{5}}\cdot \left(\frac{1+\sqrt{5}}{2}\right)^{n+1} \leq Q_n \leq \frac{2}{\sqrt{5}}\cdot \left(\frac{1+\sqrt{5}}{2}\right)^{n+1}.
\]
Combing this with (\ref{number 1}), notice that $Q_n \geq Q_{n-1}$, we have that
\begin{equation}\label{fibonacci}
A^{-1}\cdot \left(\frac{\sqrt{5}+1}{2}\right)^{-2n}  \leq\frac{1}{2Q^2_n} \leq \mathbf{P}(b_n = 1) \leq \frac{1}{Q^2_n}  = A \cdot \left(\frac{\sqrt{5}+1}{2}\right)^{-2n},
\end{equation}
where $A >1$ is an absolute constant. Since
\[
\left\{x \in [0,1): b_n(x) \leq j\right\} = \bigcup_{(b_1,b_2,\cdots,b_n) \in \Sigma_{n,\leq j}} B(b_1,b_2,\cdots,b_n)
\]
and for any $(b_1,b_2,\cdots,b_n) \in \Sigma_{n,\leq j}$,
\[
\mathbf{P}(B(b_1, b_2, \cdots, b_n) \leq \mathbf{P}(B(\underbrace{1,1,\cdots,1}_{n})),
\]
combing these with Lemma \ref{Numbers} and (\ref{fibonacci}), we deduce that
\[
\mathbf{P}\left\{x \in [0,1): b_n(x) \leq j\right\} \leq \sharp \Sigma_{n,\leq j} \cdot \mathbf{P}(B(\underbrace{1,1,\cdots,1}_{n})) \leq  A\cdot C^{j-1}_{n+j-1} \cdot \left(\frac{\sqrt{5}+1}{2}\right)^{-2n}.
\]
\end{proof}

To prove Theorem \ref{LDP}, we also need the following Lemma \ref{jixian}.

\begin{lemma}\label{jixian}
Let $\{b_n: n \geq 1\}$ be the partial quotient sequence of the ECF expansion. Then
\begin{equation*}
\lim_{n \to \infty}\frac{1}{n}\log \mathbf{E}(b_n^\theta) =
\begin{cases}
\max\left\{-2\log \frac{\sqrt{5}+1}{2},\log \frac{1}{1-\theta}\right\}, &\text{if $\theta < 1$}; \\
+\infty , & \text{if $\theta \geq 1$}.
\end{cases}
\end{equation*}
\end{lemma}

\begin{proof}
Let $\theta \geq 1$. Notice that $b_{n+1} \geq b_n$ with $b_1 \geq 1$ for all $n \geq 1$, the equation (\ref{initial distribution}) yields that for any $n \geq 1$,
\begin{equation*}
\mathbf{E}(b_n^\theta) \geq \mathbf{E}(b_1^\theta) = \sum_{k=1}^\infty \mathrm{P}(b_1 =k)\cdot k^\theta = \sum_{k=1}^\infty \frac{k^\theta}{k(k+1)} = + \infty.
\end{equation*}
Therefore, $\lim\limits_{n \to \infty}\frac{1}{n}\log E(b_n^\theta) = + \infty$.

In the following, we always assume that $\theta < 1$. Since
\[
\lim_{j \to \infty} \frac{j}{j+2} = \lim_{j \to \infty} \left(1+ \frac{1}{j}\right)\cdot \left(1-\frac{1}{j}\right)^{\theta -1} =1,
\]
Proposition \ref{lei Markov} and Lemma \ref{jichu} imply that for any $0<\varepsilon <1$, there exists positive integer $N = N(\varepsilon)$ such that for all $j > N$,
\begin{equation}\label{daxiaojie}
\frac{1-\varepsilon}{1-\theta} \leq \sum_{k=j}^{\infty} \mathbf{P}(b_{n+1}= k~|~b_n=j)\cdot \left(\frac{k}{j}\right)^\theta \leq \frac{1+\varepsilon}{1-\theta}.
\end{equation}
By the definition of expectation, we know that
\[
\mathbf{E}(b_n^\theta) = \sum_{k=1}^{\infty} \mathbf{P}(b_n = k) \cdot k^\theta = \sum_{k=1}^{N} \mathbf{P}(b_n = k) \cdot k^\theta + \sum_{k=N+1}^{\infty} \mathbf{P}(b_n = k) \cdot k^\theta.
\]
We will prove
\[
\lim_{n \to \infty}\frac{1}{n}\log \mathbf{E}(b_n^\theta) = \max\left\{-2\log \frac{\sqrt{5}+1}{2},\log \frac{1}{1-\theta}\right\}.
\]
The proof is divided into two parts:\\
{\bf Part 1. Lower bound}\\
$\bullet$ By (\ref{fibonacci}), it is clear to see that for any $n \geq 1$,
\[
\sum_{k=1}^{N} \mathbf{P}(b_n = k) \cdot k^\theta \geq \mathbf{P}(b_n = 1) = \mathbf{P}(B(1,1,\cdots,1)) \geq A^{-1}\cdot \left(\frac{\sqrt{5}+1}{2}\right)^{-2n}.
\]
$\bullet$ For any $n \geq 1$, by the definition of conditional probability, we deduce that
\begin{align}\label{expectation}
\sum_{k=N+1}^{\infty} \mathbf{P}(b_n = k) \cdot k^\theta &= \sum_{k =N+1}^\infty \sum_{j = 1}^{k}\mathbf{P}(b_n = k~|~b_{n-1} = j) \cdot \mathbf{P}(b_{n-1} = j)\cdot k^\theta \notag \\
&= \sum_{j = 1}^{N} \mathbf{P}(b_{n-1} = j) \cdot j^\theta \cdot \sum_{k =N+1}^\infty \mathbf{P}(b_n = k~|~b_{n-1} = j) \cdot \left(\frac{k}{j}\right)^\theta \notag \\
& + \sum_{j = N+1}^{\infty} \mathbf{P}(b_{n-1} = j) \cdot j^\theta \cdot \sum_{k =j}^\infty \mathbf{P}(b_n = k~|~b_{n-1} = j) \cdot \left(\frac{k}{j}\right)^\theta,
\end{align}
where the second equality is obtained by inverting the order of summations. Since the first term of the second equality in (\ref{expectation}) is nonnegative, we have that
\begin{align*}
\sum_{k=N+1}^{\infty} \mathbf{P}(b_n = k) \cdot k^\theta
& \geq \sum_{j = N+1}^{\infty} \mathbf{P}(b_{n-1} = j) \cdot j^\theta \cdot \sum_{k =j}^\infty \mathbf{P}(b_n = k~|~b_{n-1} = j) \cdot \left(\frac{k}{j}\right)^\theta \notag \\
& \geq \frac{1-\varepsilon}{1-\theta} \cdot \sum_{j = N+1}^{\infty} \mathbf{P}(b_{n-1} = j) \cdot j^\theta,
\end{align*}
where the last inequality follows from (\ref{daxiaojie}). Repeating the above procedure $(n-1)$ times, we obtain that
\begin{align*}
\sum_{k=N+1}^{\infty} \mathbf{P}(b_n = k) \cdot k^\theta
& \geq \frac{1-\varepsilon}{1-\theta} \cdot \sum_{j = N+1}^{\infty} \mathbf{P}(b_{n-1} = j) \cdot j^\theta \\
& \geq \cdots\\
& \geq \left(\frac{1-\varepsilon}{1-\theta}\right)^{n-1} \cdot \sum_{j = N+1}^{\infty} \mathbf{P}(b_1 = j) \cdot j^\theta : = M_1 \cdot \left(\frac{1-\varepsilon}{1-\theta}\right)^{n-1},
\end{align*}
where $M_1:=M_1(\varepsilon,\theta) = \sum_{j = N+1}^{\infty} \frac{j^\theta}{j(j+1)}$ is a positive constant since $\theta <1$ and it also only depends on $\varepsilon$ and $\theta$.

Therefore,
\[
\mathbf{E}(b_n^\theta) \geq A^{-1}\cdot\left(\frac{\sqrt{5}+1}{2}\right)^{-2n} + M_1 \cdot \left(\frac{1-\varepsilon}{1-\theta}\right)^{n-1}.
\]
In view of Lemma \ref{basic}, we deduce that
\begin{align*}
\liminf_{n \to \infty}\frac{1}{n}\log \mathbf{E}(b_n^\theta) &\geq \liminf_{n \to \infty}\frac{1}{n}\log \left( A^{-1} \cdot\left(\frac{\sqrt{5}+1}{2}\right)^{-2n} + M_1 \cdot \left(\frac{1-\varepsilon}{1-\theta}\right)^{n-1}\right) \\
& = \max\left\{-2\log \frac{\sqrt{5}+1}{2},\log \frac{1-\varepsilon}{1-\theta}\right\}.
\end{align*}
Since $0 < \varepsilon <1$ is arbitrary, we get
\[
\liminf_{n \to \infty}\frac{1}{n}\log \mathbf{E}(b_n^\theta) \geq \max\left\{-2\log \frac{\sqrt{5}+1}{2},\log \frac{1}{1-\theta}\right\}.
\]
{\bf Part 2. Upper bound} Let $\gamma = ((\sqrt{5}+1)/2)^{-2}$. \\
$\bullet$ In view of Lemma \ref{number}, we have that
\begin{align*}
\sum_{k=1}^{N} \mathbf{P}(b_n = k) \cdot k^\theta  \leq N\cdot\mathbf{P}(b_n \leq N)
 \leq  N \cdot A\cdot  C^{N-1}_{n+N-1}\cdot \gamma^n.
\end{align*}
$\bullet$ We first estimate the first term of the second equality in (\ref{expectation}) and then give an upper bound for the second term. Notice that
$\sum_{j = 1}^{N}\mathbf{P}(b_{n-1} = j) = \mathbf{P}(b_{n-1} \leq N) \leq A\cdot  C^{N-1}_{n+N-2}\cdot \gamma^{n-1}$, we know that
\begin{align}\label{first term}
& \sum_{j = 1}^{N} \mathbf{P}(b_{n-1} = j)\cdot j^\theta \cdot \sum_{k =N+1}^\infty \mathbf{P}(b_n = k~|~b_{n-1} = j) \cdot \left(\frac{k}{j}\right)^\theta \notag \\
\leq &\ \sum_{j = 1}^{N}\mathbf{P}(b_{n-1} = j) \cdot \sum_{k =N+1}^\infty \frac{j+1}{k(k+1)} \cdot k^\theta \notag \\
\leq &\ (N+1)\cdot\sum_{j = 1}^{N}\mathbf{P}(b_{n-1} = j)\cdot \sum_{k =N+1}^\infty \frac{k^\theta }{k(k+1)} \notag \\
\leq &\ M_1 \cdot (N+1) \cdot A\cdot  C^{N-1}_{n+N-2}\cdot \gamma^{n-1} := M_2 \cdot C^{N-1}_{n+N-2}\cdot \gamma^{n-1},
\end{align}
where the first inequality is from Proposition \ref{Markov chain}, the second inequality is from $j \leq N$ and the constant $M_2 = M_1 \cdot (N+1) \cdot A$ only depends on $\varepsilon$ and $\theta$. It follows from the right inequality in (\ref{daxiaojie}) that
\[
\sum_{j = N+1}^{\infty} \mathbf{P}(b_{n-1} = j) \cdot j^\theta \cdot \sum_{k =j}^\infty \mathbf{P}(b_n = k~|~b_{n-1} = j) \cdot \left(\frac{k}{j}\right)^\theta
\leq  \frac{1+\varepsilon}{1-\theta} \cdot \sum_{j = N+1}^{\infty} \mathbf{P}(b_{n-1} = j) \cdot j^\theta.
\]
Combing this with (\ref{expectation}) and (\ref{first term}), we obtain that
\begin{align}\label{second term 2}
\sum_{k=N+1}^{\infty} \mathbf{P}(b_n = k) \cdot k^\theta \leq M_2 \cdot C^{N-1}_{n+N-2}\cdot \gamma^{n-1} +
\frac{1+\varepsilon}{1-\theta} \cdot \sum_{j = N+1}^{\infty} \mathbf{P}(b_{n-1} = j) \cdot j^\theta.
\end{align}
Similarly, we have that
\[
\sum_{j = N+1}^{\infty} \mathbf{P}(b_{n-1} = j) \cdot j^\theta \leq M_2 \cdot C^{N-1}_{n+N-3}\cdot \gamma^{n-2}+
\frac{1+\varepsilon}{1-\theta} \cdot \sum_{j = N+1}^{\infty} \mathbf{P}(b_{n-2} = j) \cdot j^\theta.
\]
In view of (\ref{second term 2}), we deduce that
\begin{align}\label{second term 3}
& \sum_{k=N+1}^{\infty} \mathbf{P}(b_n = k) \cdot k^\theta \notag \\
\leq&\ M_2 \cdot C^{N-1}_{n+N-2}\cdot \gamma^{n-1} + \frac{1+\varepsilon}{1-\theta}\cdot M_2 \cdot C^{N-1}_{n+N-3}\cdot \gamma^{n-2} + \left(\frac{1+\varepsilon}{1-\theta}\right)^2 \cdot \sum_{j = N+1}^{\infty} \mathbf{P}(b_{n-2} = j) \cdot j^\theta \notag \\
\leq &\ M_2 \cdot C^{N-1}_{n+N-2} \cdot \left( \gamma^{n-1} + \gamma^{n-2}\cdot \frac{1+\varepsilon}{1-\theta}\right) + \left(\frac{1+\varepsilon}{1-\theta}\right)^2 \cdot \sum_{j = N+1}^{\infty} \mathbf{P}(b_{n-2} = j) \cdot j^\theta,
\end{align}
where the last inequality is from $C^{N-1}_{n+N-3}  \leq C^{N-1}_{n+N-2}$. Iterating the process in (\ref{second term 3}), we know that
\begin{align}\label{second term 4}
\sum_{k=N+1}^{\infty} \mathbf{P}(b_n = k) \cdot k^\theta & \leq M_2 \cdot C^{N-1}_{n+N-2} \cdot \sum_{m=0}^{n-2} \gamma^{n-m} \left(\frac{1+\varepsilon}{1-\theta}\right)^m + M_1 \cdot \left(\frac{1+\varepsilon}{1-\theta}\right)^{n-1} \notag \\
& \leq M_2 \cdot C^{N-1}_{n+N-2} \cdot \sum_{m=0}^{n-1} \gamma^{n-m} \left(\frac{1+\varepsilon}{1-\theta}\right)^m
\end{align}
where the first inequality is from the definition of $M_1$ and the last inequality follows from $C^{N-1}_{n+N-2} \geq 1$ and the definition of $M_2$. It is not difficult to check that the geometric series
\[
\sum_{m=0}^{n-1} \gamma^{n-m} \left(\frac{1+\varepsilon}{1-\theta}\right)^m = M_3 \cdot \left(\left( \frac{1+\varepsilon}{1-\theta}\right)^{n}- \gamma^n \right),
\]
where $M_3$ is the constant only depending on $\varepsilon$ and $\theta$. Combing this with (\ref{second term 4}), we have that
\[
\sum_{k=N+1}^{\infty} \mathbf{P}(b_n = k) \cdot k^\theta \leq M \cdot C^{N-1}_{n+N-2} \cdot \left(\left( \frac{1+\varepsilon}{1-\theta}\right)^{n}- \gamma^n \right),
\]
where $M = M_2\cdot M_3 $ is a constant only depending on $\varepsilon$ and $\theta$.

Therefore,
\begin{align*}
\mathbf{E}(b_n^\theta) &= \sum_{k=1}^{N} \mathbf{P}(b_n = k) \cdot k^\theta + \sum_{k=N+1}^{\infty} \mathbf{P}(b_n = k) \cdot k^\theta \\
& \leq N \cdot A\cdot  C^{N-1}_{n+N-1}\cdot \gamma^n + M \cdot C^{N-1}_{n+N-2} \cdot \left(\left( \frac{1+\varepsilon}{1-\theta}\right)^{n}- \gamma^n \right) \\
& =  H(n) \cdot\left(\frac{\sqrt{5}+1}{2}\right)^{-2n} + M \cdot C^{N-1}_{n+N-2} \cdot \left(\frac{1+\varepsilon}{1-\theta}\right)^{n}.
\end{align*}
where $C^{N-1}_{n+N-2}$ and $H(n)$ are both polynomials of $n$. Consequently, we deduce that
\begin{align*}
\limsup_{n \to \infty}\frac{1}{n}\log \mathbf{E}(b_n^\theta) &\leq \limsup_{n \to \infty}\frac{1}{n}\log \left( H(n) \cdot\left(\frac{\sqrt{5}+1}{2}\right)^{-2n} + M \cdot C^{N-1}_{n+N-2} \cdot \left(\frac{1+\varepsilon}{1-\theta}\right)^{n}\right)\\
& = \max\left\{-2\log \frac{\sqrt{5}+1}{2},\log \frac{1+\varepsilon}{1-\theta}\right\},
\end{align*}
where the last equality is obtained by observing that $C^{N-1}_{n+N-2}$ and $H(n)$ are both polynomials of $n$. Let $\varepsilon \to 0^+$, we obtain that
\[
\limsup_{n \to \infty}\frac{1}{n}\log \mathbf{E}(b_n^\theta) \leq \max\left\{-2\log \frac{\sqrt{5}+1}{2},\log \frac{1}{1-\theta}\right\}.
\]
\end{proof}

\begin{remark}\label{vs}
Since the digit sequence $\{q_n:n \geq 1\}$ (the notation follows Zhu \cite{lesZhu2014}) of Engel expansions is also non-decreasing and forms a forms a homogeneous Markov chain (see Erd\H{o}s \cite{lesE.R.S58}), we know that the similar result of Lemma \ref{number} in the setting of Engel expansions can be obtained
\begin{equation}\label{log2}
\mathrm{P}(q_n =2) = 2^{-n}\ \ \ \ \text{and}\ \ \ \ \mathrm{P}(q_n \leq j) \leq C^{j-1}_{n+j-1} \cdot 2^{-n}.
\end{equation}
The above method in Lemma \ref{jixian} immediately yields that
\begin{equation*}
\lim_{n \to \infty}\frac{1}{n}\log \mathbf{E}(q_n^\theta) =
\begin{cases}
\max\left\{-\log 2,\log \frac{1}{1-\theta}\right\}, &\text{if $\theta < 1$}; \\
+\infty , & \text{if $\theta \geq 1$},
\end{cases}
\end{equation*}
where the quantity $\log2$ is from (\ref{log2}).
This limit plays an important role in the proof of Zhu (see \cite[Lemma 1.1]{lesZhu2014}). However, the digit sequence $\{d_n:n \geq 1\}$ of modified Engel continued fractions or alternating Engel expansions is increasing. The similar arguments of Lemmas \ref{number} and \ref{jixian} imply that
\[
\mathrm{P}(d_n =n) = \frac{1}{n!}\ \ \ \ \text{and}\ \ \ \ \mathrm{P}(q_n \leq j) \leq C^{j-1}_{n+j-1} \cdot \frac{1}{n!}.
\]
Note that $\lim\limits_{n \to \infty}(-\log n!)/n = -\infty$, so we have
\begin{equation*}
\lim_{n \to \infty}\frac{1}{n}\log \mathbf{E}(d_n^\theta) =
\begin{cases}
\log \left(\frac{1}{1-\theta}\right), &\text{if $\theta < 1$}; \\
+\infty , & \text{if $\theta \geq 1$}.
\end{cases}
\end{equation*}
Combing this with G\"{a}rtner-Ellis theorem, we obtain the large deviations for modified Engel continued fractions and alternating Engel expansions.
\end{remark}

Now we are ready to prove Theorem \ref{LDP}.

\begin{proof}[Proof of Theorem \ref{LDP}]
For any $\theta \in \mathbb{R}$, we define the function as
\[
\Lambda_n(\theta) = \log \mathbf{E}\left(\exp{\left(\frac{\log b_n -n}{n}\cdot \theta\right)}\right).
\]
By Lemma \ref{jixian}, we know that the pressure function is given as
\[
\Lambda(\theta) = \lim_{n \to \infty}\frac{1}{n}\Lambda_n\left(n\theta\right) =
\begin{cases}
-\theta - 2\log\frac{\sqrt{5}+1}{2}, &\text{if $\theta \leq -\frac{\sqrt{5}+1}{2}$}; \\
-\theta - \log(1-\theta), &\text{if $-\frac{\sqrt{5}+1}{2} <\theta < 1$}; \\
+\infty,  & \text{if $\theta \geq 1$}.
\end{cases}
\]
By G\"{a}rtner-Ellis theorem (see \cite[Theorem 2.3.6]{lesD.Z1998}), we obtain the sequence $\left\{\frac{\log b_n -n}{n}: n\geq 1\right\}$ satisfies a LDP with speed $n$ and good rate function
\[
I(x) = \sup_{\theta \in \mathbb{R}}\left\{\theta x - \Lambda(\theta)\right\}\ \text{for all}\ x \in \mathbb{R}.
\]
Now we show that $I(x)$ is the same as (\ref{LDP equation}). Notice that $\theta x - \Lambda(\theta)= -\infty$ for all $x \in \mathbb{R}$ when $\theta \geq 1$, so we only need to compute the followings
\[
\Phi(x) = \sup_{-\frac{\sqrt{5}+1}{2} < \theta < 1}\big\{\theta x +\theta+ \log(1-\theta)\big\}\ \text{and}\
\Psi(x) = \sup_{\theta \leq -\frac{\sqrt{5}+1}{2}}\left\{\theta x +\theta+ 2\log\frac{\sqrt{5}+1}{2}\right\}.
\]
Thus, $I(x) = \max\{\Phi(x), \Psi(x)\}$ for all $x \in \mathbb{R}$. We first give that
\[
\Phi(x) =
\begin{cases}
x - \log(x+1), &\text{if $x > -\frac{\sqrt{5}-1}{2}$}; \\
-\frac{\sqrt{5}+1}{2}(x +1) + 2\log\frac{\sqrt{5}+1}{2},  & \text{if $x \leq -\frac{\sqrt{5}-1}{2}$}.
\end{cases}
\]
Let $f(\theta) = \theta x + \theta+ \log(1-\theta)$ for any $-(\sqrt{5}+1)/2< \theta < 1$. It is clear to check $f$ is strictly concave and that $\theta = x/(x+1)$ is the unique maximal point point of it. From the inequalities $ -(\sqrt{5}+1)/2 < x/(x+1) < 1$, we know that $x > - (\sqrt{5}-1)/2$. That is to say, when $x > - (\sqrt{5}-1)/2$, the function $f$ reaches the maximal value at the point $x/(x+1) \in (-(\sqrt{5}+1)/2, 1)$, i.e., $\Phi(x) = x - \log(x+1)$ in this case. Here we claim that $f$ is decreasing on the interval $(-(\sqrt{5}+1)/2, 1)$ when $x \leq - (\sqrt{5}-1)/2$. In fact, it follows from $-1 <x \leq - (\sqrt{5}-1)/2$ that the maximal point $x/(x+1) \leq -(\sqrt{5}+1)/2$. That is to say, the maximal point $x/(x+1)$ locates in the left of the interval $(-(\sqrt{5}+1)/2, 1)$. In other words, the function $f$ is decreasing on the interval $(-(\sqrt{5}+1)/2, 1)$. If $x \leq -1$, we know that $f(\theta) = \theta (x +1) + \log(1-\theta)$ is decreasing. As is mentioned above, the function $f$ reaches the maximal value at the point $-(\sqrt{5}+1)/2$. Therefore, $\Phi(x)$ is completely determined for any $x \in \mathbb{R}$. Next, we show that
\[
\Psi(x) =
\begin{cases}
-\frac{\sqrt{5}+1}{2}(x +1) + 2\log\frac{\sqrt{5}+1}{2}, &\text{if $x \geq -1$}; \\
+ \infty,  & \text{if $x < -1$}.
\end{cases}
\]
Let $g(\theta) = \theta x +\theta+ 2\log((\sqrt{5}+1)/2)$ for any $\theta \leq -(\sqrt{5}+1)/2$. In fact, when $x < -1$, the function $g$ is decreasing and hence that $\Psi(x) = + \infty$ in this case. If $x >-1$, the function $g$ is increasing and hence that $g$ reaches the maximal value at $-(\sqrt{5}+1)/2$. It is easy to see $g$ is constantly $2\log((\sqrt{5}+1)/2)$ when $x =-1$. Now we claim that
\begin{equation*}
I(x) = \sup_{\theta \in \mathbb{R}}\left\{\theta x - \Lambda(\theta)\right\} =
\begin{cases}
x - \log(x +1), &\text{if $x > -\frac{\sqrt{5}-1}{2}$}; \\
-\frac{\sqrt{5}+1}{2}(x +1) + 2\log\frac{\sqrt{5}+1}{2}, &\text{if $-1\leq x \leq -\frac{\sqrt{5}-1}{2}$}; \\
+\infty,  & \text{if $x < -1$}.
\end{cases}
\end{equation*}
Since $I(x) = \max\{\Phi(x), \Psi(x)\}$ for all $x \in \mathbb{R}$, by the definitions of $\Phi(x)$ and $\Psi(x)$, we easily determine the rate function $I(x)$ when $-1\leq x \leq -(\sqrt{5}-1)/2$ and $x < -1$. It remains to compare $\Phi(x)$ with $\Psi(x)$ when $x >-(\sqrt{5}-1)/2$. To do this, define the function
\[
h(x) = x - \log(x +1) + \frac{\sqrt{5}+1}{2}(x +1)- 2\log\frac{\sqrt{5}+1}{2}.
\]
Then we have that $h(-(\sqrt{5}-1)/2) =0$ and the derivative of $h$ satisfies $h^\prime(x) = x/(x+1) + (\sqrt{5}+1)/2$ and $h^\prime(x) > 0$ if $x >-(\sqrt{5}-1)/2$. Thus, $h(x) \geq 0$ when $x >-(\sqrt{5}-1)/2$. Therefore, the rate function $I(x)$ is completely established.

\end{proof}

\subsection{Proof of moderate deviation principle}

As applications of Proposition \ref{Markov chain} and Lemma \ref{jichu}, we obtain the following lemma.

\begin{lemma}\label{MDPlemma}
Let $n \in \mathbb{N}$ and $-1/2<\theta <1/2$. Then for any $j\geq 3n-1$, we have
\[
\sum_{k=j}^{\infty} \mathbf{P}(b_{n+1}= k~|~b_n=j)\cdot \left(\frac{k}{j}\right)^\theta  \leq  \frac{1}{1-n^{-1}}\cdot \frac{1}{1-\theta}.
\]
\end{lemma}

\begin{proof}
For any $j >1$, it follows from Lemma \ref{jichu} that
\[
\sum_{k=j}^{\infty} \frac{j+1}{k(k+1)} \left(\frac{k}{j}\right)^\theta \leq  \left(1+ \frac{1}{j}\right)\cdot \left(1-\frac{1}{j}\right)^{\theta -1}\cdot \frac{1}{1-\theta}.
\]
Combing this with Proposition \ref{Markov chain}, we deduce that
\[
\sum_{k=j}^{\infty} \mathbf{P}(b_{n+1}= k~|~b_n=j)\cdot \left(\frac{k}{j}\right)^\theta  \leq  \left(1+ \frac{1}{j}\right)\cdot \left(1-\frac{1}{j}\right)^{\theta -1}\cdot \frac{1}{1-\theta}.
\]
Now it is sufficient to prove that
\[
 \left(1+ \frac{1}{j}\right)\cdot \left(1-\frac{1}{j}\right)^{\theta -1} \leq \frac{1}{1-n^{-1}}
\]
for any $j\geq 3n-1$. Let $n \in \mathbb{N}$. In fact, since $-1/2<\theta <1/2$, we deduce that
\begin{equation}\label{MDPeq}
\left(1+ \frac{1}{j}\right)\cdot \left(1-\frac{1}{j}\right)^{\theta -1} \leq \frac{1+1/j}{(1-1/j)^{3/2}}.
\end{equation}
Note that $(1-1/j)^{3/2} \geq 1-2/j$ for any $j \geq 1$, we have
\[
\frac{(1-1/j)^{3/2}}{1+1/j} \geq \frac{1-2/j}{1+1/j} =  1- \frac{3}{j+1} \geq 1 - \frac{1}{n}
\]
for any $j\geq 3n-1$. Combing this with (\ref{MDPeq}), we complete the proof.
\end{proof}

Now we are ready to prove Theorem \ref{MDP}.

\begin{proof}[Proof of Theorem \ref{MDP}]
Let $\{a_n: n \geq 1\}$ be the sequence of positive numbers satisfying the conditions in (\ref{a-n}).
For any $\lambda \in \mathbb{R}$, we consider the logarithmic moment generating function (see Dembo and Zeitouni \cite[Section 2.3]{lesD.Z1998}) of $\frac{\log b_n -n}{a_n}$,
\[
\Lambda_n(\lambda) = \log \mathbf{E}\left(\exp\left(\lambda\cdot\frac{\log b_n -n}{a_n}\right)\right).
\]
From the G\"{a}rtner-Ellis theorem, in order to obtaining the desired result, it suffices to show that for any $\lambda \in \mathbb{R}$,
\[
\Lambda(\lambda) = \lim_{n \to \infty} \frac{n}{a_n^2} \Lambda_n\left(\frac{a_n^2}{n}\lambda\right) = \frac{\lambda^2}{2}.
\]
That is,
\begin{equation}\label{desired result}
\lim_{n \to \infty} \frac{n}{a_n^2} \log \mathbf{E}\left(\exp\left\{\frac{a_n}{n}\left(\log b_n -n
\right)\lambda\right\}\right)= \frac{\lambda^2}{2}.
\end{equation}

For any $\lambda \in \mathbb{R}$ and $n \geq 1$, let
\[
\theta_n:=\theta_n(\lambda) = \frac{a_n}{n}\lambda\ \ \ \  \text{and}\ \ \ \  \Upsilon_n(\lambda) = \mathbf{E}\big(\exp\{\theta_n(\log b_n -n)\}\big).
\]
In view of (\ref{a-n}), it is clear that $\theta_n \to 0$ as $n \to \infty$ and $\Upsilon_n(\lambda)$ can be rewritten as
\begin{equation}\label{Upsilon-n}
\Upsilon_n(\lambda) = e^{-n\theta_n}\mathbf{E}(b_n^{\theta_n}).
\end{equation}
Being similar to the proof of large deviation principle part,
to get (\ref{desired result}), we only need to estimate the expectation $\mathbf{E}(b_n^{\theta_n})$. Since $\theta_n \to 0$ as $n \to \infty$, there exists a positive number $N$ (only depending on $\lambda$) such that for all $n \geq N$, we have $-1/2<\theta_n < 1/2$. In the following, we always fix such $n$. Now we will give the lower and upper bounded estimates of $\mathbf{E}(b_n^{\theta_n})$ respectively.

We first give the lower bound for $\mathbf{E}(b_n^{\theta_n})$. Being similar to the Part 1 in the proof of Lemma \ref{jixian}, we know that
\begin{align*}
\sum_{k=1}^\infty \mathbf{P} (b_n =k)\cdot k^{\theta_n}
& = \sum_{j = 1}^{\infty} \mathbf{P}(b_{n-1} = j) \cdot j^{\theta_n} \cdot \sum_{k =j}^\infty \mathbf{P}(b_n = k~|~b_{n-1} = j) \cdot \left(\frac{k}{j}\right)^{\theta_n} \notag \\
& \geq \sum_{j = 1}^{\infty} \mathbf{P}(b_{n-1} = j) \cdot j^{\theta_n} \cdot \sum_{k =j}^\infty \frac{j}{k(k+2)} \left(\frac{k}{j}\right)^{\theta_n},
\end{align*}
where the last inequality follows from Proposition \ref{Markov chain}. Combing this with Lemma \ref{jichu}, we have that
\begin{align*}
\sum_{k=1}^\infty \mathbf{P} (b_n =k)\cdot k^{\theta_n} & \geq \sum_{j = 1}^{\infty} \mathbf{P}(b_{n-1} = j) \cdot j^{\theta_n} \cdot  \frac{j}{j+2} \cdot \frac{1}{1-\theta_n}\\
& \geq \frac{1}{1+2} \cdot \frac{1}{1-\theta_n}\cdot \sum_{j = 2}^{\infty} \mathbf{P}(b_{n-1} = j) \cdot j^{\theta_n}
\end{align*}
since $j/(j+2) \geq 1/(1+2)$ for any $j \geq 1$. Thus, we deduce that
\[
\sum_{k=1}^\infty \mathbf{P} (b_n =k)\cdot k^{\theta_n} \geq \frac{1}{1+2} \cdot \frac{1}{1-\theta_n}\cdot \sum_{j = 2}^{\infty} \mathbf{P}(b_{n-1} = j) \cdot j^{\theta_n}.
\]
Similarly, we obtain that
\[
\sum_{k = 2}^{\infty} \mathbf{P}(b_{n-1} = k) \cdot k^{\theta_n} \geq \frac{2}{2+2} \cdot \frac{1}{1-\theta_n}\cdot \sum_{j = 3}^{\infty} \mathbf{P}(b_{n-2} = j) \cdot j^{\theta_n}.
\]
Repeating the above procedure, by (\ref{initial distribution}) and $-1/2< \theta_n < 1/2$, we actually have that
\begin{align*}
\sum_{k=1}^\infty \mathbf{P} (b_n =k)\cdot k^{\theta_n} &\geq \left(\frac{1}{3}\cdot \frac{2}{4}\cdot \cdots \cdot \frac{n-1}{n+1}\right)\cdot \left(\frac{1}{1-\theta_n}\right)^{n-1} \cdot \sum_{j = n}^{\infty} \mathbf{P}(b_1 = j) \cdot j^{\theta_n} \\
&\geq M(n) \cdot \frac{2}{n(n+1)} \cdot \left(\frac{1}{1-\theta_n}\right)^{n-1},
\end{align*}
where $M(n) = \sum_{j = n}^{\infty}\frac{1}{(j+1)j^{3/2}}$. Therefore,
\begin{equation}\label{MDPL}
\mathbf{E}(b_n^{\theta_n}) \geq M(n)\cdot \frac{2}{n(n+1)} \cdot \left(\frac{1}{1-\theta_n}\right)^{n-1}.
\end{equation}
Noth that $j(j+1) \leq (j+1)j^{3/2}$ for any $j \geq 1$, we obtain that
\[
\frac{1}{(n+1)n^{3/2}}\leq M(n) \leq \frac{1}{n},
\]
which implies that
\[
\lim_{n \to \infty} \frac{n}{a^2_n} \log M(n)=0
\]
since $\lim_{n \to \infty} (n\log n)/a^2_n =0$ in view of (\ref{a-n}).
Combing this with (\ref{Upsilon-n}) and (\ref{MDPL}), by Taylor formula, we actually deduce that
\begin{align*}
\liminf_{n \to \infty} \frac{n}{a_n^2} \log \mathbf{E}\left(\exp\left\{\frac{a_n}{n}\left(\log b_n -n
\right)\lambda\right\}\right)
\geq \liminf_{n \to \infty} \left( \frac{n^2}{a_n^2}(-\theta_n) +  \frac{n^2}{a_n^2}\log \frac{1}{1-\theta_n} \right) = \frac{\lambda^2}{2}.
\end{align*}

Next we give the upper bound for $\mathbf{E}(b_n^{\theta_n})$. Being similar to the Part 2 in the proof of Lemma \ref{jixian}, let $1+\varepsilon = 1/(1-n^{-1})$ and $N:= N(n)=3n-1$. Lemma \ref{MDPlemma} guarantees that the methods of the Part 2 in the proof of Lemma \ref{jixian} are still valid. Therefore, we actually obtain that
\[
\mathbf{E}(b_n^{\theta_n}) \leq   U(n) \cdot\left(\frac{\sqrt{5}+1}{2}\right)^{-2n} + V(n) \cdot \left(\frac{1}{1-n^{-1}}\cdot \frac{1}{1-\theta_n}\right)^{n},
\]
where $U(n)$ and $V(n)$ both are polynomials of $n$ with some degree. Note that

\[
\lim_{n \to \infty} \frac{n}{a^2_n} \log U(n) = \lim_{n \to \infty} \frac{n}{a^2_n} \log V(n) =\lim_{n \to \infty} \frac{n^2}{a^2_n} \log \left(1-\frac{1}{n}\right) =0
\]
and
\[
\lim_{n \to \infty} \frac{n}{a^2_n} \log \left(\frac{\sqrt{5}+1}{2}\right)^{-2n} = - 2 \lim_{n \to \infty} \frac{n^2}{a^2_n} \log \left(\frac{\sqrt{5}+1}{2}\right) = -\infty,
\]
by Taylor formula, we eventually deduce that
\begin{align*}
&\limsup_{n \to \infty} \frac{n}{a_n^2} \log \mathbf{E}\left(\exp\left\{\frac{a_n}{n}\left(\log b_n -n
\right)\lambda\right\}\right) \\
\leq & \limsup_{n \to \infty}  \left\{\frac{n^2}{a_n^2}(-\theta_n)- \frac{n^2}{a^2_n} \log \left(1-\frac{1}{n}\right) - \frac{n^2}{a^2_n} \log \left(1-\theta_n\right)\right\}
= \frac{\lambda^2}{2}.
\end{align*}
Thus, the equality (\ref{desired result}) is established. By the G\"{a}rtner-Ellis theorem, we obtain the sequence $\left\{\frac{\log b_n -n}{a_n}: n\geq 1\right\}$ satisfies an MDP with speed $n^{-1}a_n^2$ and good rate function
\[
J(x) = \sup_{\lambda \in \mathbb{R}}\left\{\lambda x - \Lambda(\lambda)\right\} = x^2/2
\]
for any $x \in \mathbb{R}$.
\end{proof}

\section{Conclusions}
As we know, fractal and large deviations are two main tools of dealing with almost everywhere results. The former is concerned with the fractal structure (such as, Hausdorff dimension) of the set of points for which the desired almost everywhere result does not hold or attains any other values; while the latter considers the speeds of probabilities of the events that the desired almost everywhere result deviates away from its ergodic mean. They are the two different aspects of describing the same thing.
What is interesting is that these two different ways have some relations in some special cases (see Denker and Kesseb\"{o}hmer \cite{lesDK01}, Fang et al.~\cite{lesFWLa}, Pesin and Weiss \cite{lesPW01}).

Engel continued fractions and modified Engel continued fractions (see \cite{lesFangSPL}) are two typical examples of Oppenheim continued fractions introduced by Fan et al.~\cite{lesFWW07} in 2007. And modified Engel continued fractions is just a small modification of Engel continued fractions. Fan et al.~\cite{lesFWW07} have shown that these two continued fractions share the classical limit theorems, such as the law of large numbers, the central limit theorem, the law of the iterated logarithm, and other statical laws. From the fractal points of view, we can obtain that for any $\alpha \geq 0$, the set
\[
\left\{x \in (0,1]:\lim_{n \to \infty}\frac{\log e_n(x)}{n} =\alpha\right\}
\]
has full Hausdorff dimension following the idea of \cite{lesKW04, lesLW01},
where $\{e_n(x): n\geq 1\}$ is the partial quotient sequence of ECF expansions or modified ECF expansions. However, comparing the Theorem 1.1 of \cite{lesFangSPL} with our Theorem \ref{LDP} in Section 2, we know that ECF expansions and modified ECF expansions
have a difference in the context of large deviations. From this point of view, large deviation is a more useful tool than other tools in such questions.




\begin{thebibliography}{10}

\bibitem{lesAF84} R. Adler and L. Flatto, {\it The backward continued fraction map and geodesic flow}, Ergodic Theory Dynam. Systems 4 (1984), no. 4, 487--492.







\bibitem{lesBCR06} R. Bass, X. Chen and J. Rosen, {\it Moderate deviations and laws of the iterated logarithm for the renormalized self-intersection local times of planar random walks}, Electron. J. Probab. 11 (2006), no. 37, 993--1030.








\bibitem{lesChen07} X. Chen, {\it Moderate deviations and laws of the iterated logarithm for the local times of additive L\'{e}vy processes and additive random walks}, Ann. Probab. 35 (2007), no. 3, 954--1006.





\bibitem{lesDKL15} K. Dajani, C. Kraaikamp and N. Langeveld, {\it Continued fraction expansions with variable numerators}, Ramanujan J. 37 (2015), no. 3, 617--639.





\bibitem{lesDKS09} K. Dajani, C. Kraaikamp and W. Steiner, {\it Metrical theory for $\alpha$-Rosen fractions}, J. Eur. Math. Soc. (JEMS) 11 (2009), no. 6, 1259--1283.



\bibitem{lesD.Z1998} A. Dembo and O. Zeitouni, {\it Large Deviations Techniques and Applications}, Springer-Verlag, 2nd Edition, New York, 1998.



\bibitem{lesDK01} M. Denker and M. Kesseb\"{o}hmer, {\it Thermodynamic formalism, large deviation, and multifractals}, Stochastic climate models (Chorin, 1999), 159--169, Progr. Probab., 49, Birkh\"{a}user, Basel, 2001.




\bibitem{lesEll80} P. Elliott, {\it Probabilistic Number Theory (I, II)}, Springer-Verlag, New York-Berlin, 1980.






\bibitem{lesE.R.S58} P. Erd\H{o}s, A. R\'{e}nyi and P. Sz\"{u}sz, {\it On Engel's and Sylvester's series}, Ann. Univ. Sci. Budapest. E\"{o}tv\"{o}s. Sect. Math. 1 (1958), 7--32.



\bibitem{lesFWW07} A.-H. Fan, B.-W. Wang and J. Wu, {\it Arithmetic and metric properties of Oppenheim continued fraction expansions}, J. Number Theory 127 (2007), no. 1, 64--82.




\bibitem{lesFangSPL} L. Fang, {\it Large and moderate deviations for modified Engel continued fractions}, Statist. Probab. Lett. 98 (2015), 98--106.


\bibitem{lesFangJNT} L. Fang, {\it Large and moderate deviation principles for alternating Engel expansions}, J. Number Theory 156 (2015), 263--276.




 \bibitem{lesFW} L. Fang and M. Wu, {\it A note on R\'{e}nyi's ``record'' problem and Engel's series}, 7 pages, arXiv:1607.03173.



\bibitem{lesFWLa} L. Fang, M. Wu and B. Li, {\it Beta-expansion and continued fraction expansion of real numbers}, 16 pages, arXiv:1603.01081.



 \bibitem{lesFMN} V. F\'{e}ray, P. M\'{e}liot and A. Nikeghbali, {\it Mod-phi convergence I: Normality zones and precise deviations}, 103 pages, arXiv:1304.2934.







\bibitem{lesGal76} J. Galambos, {\it Representations of Real Numbers by Infinite Series}, Lecture Notes in Mathematics, Vol. 502. Springer-Verlag, Berlin-New York, 1976.






\bibitem{lesGao08} F.-Q. Gao, {\it Moderate deviations and law of the iterated logarithm in $L_1(\mathbb{R}^d)$ for kernel density estimators}, Stochastic Process. Appl. 118 (2008), no. 3, 452--473.




\bibitem{lesGM15}R. Giuliano and C. Macci, {\it Asymptotic results for weighted means of random variables which converge to a Dickman distribution, and some number theoretical applications}, ESAIM Probab. Stat. 19 (2015), 395--413.






\bibitem{lesHK01} D. Hardcastle and K. Khanin, {\it Continued fractions and the $d$-dimensional Gauss transformation}, Comm. Math. Phys. 215 (2001), no. 3, 487--515.









\bibitem{lesHKS02} Y. Hartono, C. Kraaikamp and F. Schweiger, {\it Algebraic and ergodic properties of a new continued fraction algorithm with non-decreasing partial quotients}, J. Th\'{e}or. Nombres Bordeaux 14 (2002), no. 2, 497--516.


\bibitem{lesHu2015} W. Hu, {\it Moderate deviation principles for Engel's, Sylvester's series and Cantor's products}, Statist. Probab. Lett. 96 (2015) 247--254.







\bibitem{lesKhi64} A. Khintchine, {\it Continued Fractions}, The University of Chicago Press, Chicago-London, 1964.


\bibitem{lesKW04} C. Kraaikamp and J. Wu, {\it On a new continued fraction expansion with non-decreasing partial quotients}, Monatsh. Math. 143 (2004), no. 4, 285--298.



\bibitem{lesLW01} Y.-Y. Liu and J. Wu, {\it Hausdorff dimensions in Engel expansions}, Acta Arith. 99 (2001), no. 1, 79--83.


\bibitem{lesLur83} J. L\"{u}roth, {\it Ueber eine eindeutige Entwickelung von Zahlen in eine unendliche Reihe}, Math. Ann. 21 (1883), 411--423.




\bibitem{lesMPB10} A. Masmoudi, W. Puech and M. Bouhlel, {\it An efficient PRBG based on chaotic map and Engel continued fractions}, J. Software Engineering and Applications 3 (2010), no. 12, 1141--1147.


\bibitem{lesMBP12} A. Masmoudi, M. Bouhlel and W. Puech, {\it Image encryption using chaotic standard map and Engle continued fractions map}, 6th International Conference on Sciences of Electronics, Technologies of Information and Telecommunications (SETIT), 2012, 474--480.




\bibitem{lesMZ16} B. Mehrdad and L. Zhu, {\it Moderate and large deviations for the Erd\H{o}s-Kac theorem}, Q. J. Math. 67 (2016), no. 1, 147--160.

\bibitem{lesMZar} B. Mehrdad and L. Zhu, {\it Limit theorems for empirical density of greatest common divisors}, to appear in Mathematical Proceedings of the Cambridge Philosophical Society.







\bibitem{lesNN02} H. Nakada and R. Natsui, {\it Some metric properties of $\alpha$-continued fractions}, J. Number Theory 97 (2002), no. 2, 287--300.




\bibitem{lesPW01} Y. Pesin and H. Weiss, {\it The multifractal analysis of Birkhoff averages and large deviations}, Global analysis of dynamical systems, 419--431, Inst. Phys., Bristol, 2001.




\bibitem{lesRad09} M. Radziwill, {\it On large deviations of additive functions}, 90 pages, arXiv:0909.5274.




\bibitem{lesRen57} A. R\'{e}nyi, {\it Representations for real numbers and their ergodic properties}, Acta Math. Acad. Sci. Hungar. 8 (1957), 477--493.




\bibitem{lesRen58} A. R\'{e}nyi, {\it Probablistic methods in number theory}, Adv. Math. (China) 4 (1958), 465--510.



\bibitem{lesRos54} D. Rosen, {\it A class of continued fractions associated with certain properly discontinuous groups}, Duke Math. J. 21, (1954). 549--563.












\bibitem{lesSha86} J. Shallit, {\it Metric theory of Pierce expansions}, Fibonacci Quart. 24 (1986), no. 1, 22--40.






\bibitem{lesTen95} G. Tenenbaum, {\it Introduction to Analytic and Probabilistic Number Theory}, Cambridge Studies in Advanced Mathematics, Cambridge University Press, Cambridge, 1995.






\bibitem{lesTou09} H. Touchette, {\it The large deviation approach to statistical mechanics}, Phys. Rep. 478 (2009), 1--69.



\bibitem{lesVar1984} S.R.S. Varadhan, {\it Large Deviations and Applications}, SIAM, Philadelphia, 1984.


\bibitem{lesV.B.P99} P. Viader, L. Bibiloni and J. Parad\'{\i}s, {\it On a problem of Alfr\'{e}d R\'{e}nyi}, Acta Arith. 91 (1999), no. 2, 107--115.


\bibitem{lesWil73} D. Williams, {\it On R\'{e}nyi's ``record" problem and Engel's series}, Bull. London Math. Soc. 5 (1973), 235-237.


\bibitem{lesZhu2014} L. Zhu, {\it On the large deviations for Engel's, Sylvester's series and Cantor's products}, Electron. Commun. Probab. 19 (2014), 1--9.
\end{thebibliography}
\end{document}